\documentclass{birkjour}
 \newtheorem{thm}{Theorem}[section]
 
 \newtheorem{lem}[thm]{Lemma}
 
 \theoremstyle{definition}
 
 \theoremstyle{remark}
 \newtheorem{rem}[thm]{Remark}
 \newtheorem*{ex}{Example}
 \numberwithin{equation}{section}
 \usepackage{amssymb}
 \theoremstyle{definition}
\newtheorem{definition}{Definition}[section]

\usepackage{amsmath, amsthm, amscd, amsfonts, amssymb, graphicx, color}
\usepackage[bookmarksnumbered, colorlinks, plainpages]{hyperref}
\begin{document}

%-------------------------------------------------------------------------
% editorial commands: to be inserted by the editorial office
%
%\firstpage{1} \volume{228} \Copyrightyear{2004} \DOI{003-0001}
%\seriesextraline{This is the Concrete Title of this Book\br H.E. R and S.T.C. W, Eds.}
%\firstpage{1}

\title{Blending type Approximations by Kantorovich variant of $\alpha$-Schurer operators}

%----------Author 2
%----------Author 1

\author[Rao]{\href{https://orcid.org/0000-0002-5681-9563}{Nadeem Rao}}
\address{
Department of Mathematics, Faculty of Science\\
\textbf{Jamia Millia Islamia}\\
New Delhi-110025\\
Delhi\\
\textbf{India}}
\email{nadeemrao1990@gmail.com}

\author[Rani]{\href{https://orcid.org/0000-0002-3075-3957}{Mamta Rani}}

\address{Department of Mathematics\\
\textbf{Shree Guru Gobind Singh Tricentenary University}\\
Gurugram-122505, Haryana\\
\textbf{India}}
\email{mamtakalra84@gmail.com}

\author[Kili\c{c}man]{\href{https://orcid.org/0000-0002-1217-963X}{Adem Kili\c{c}man}}
\address{
	Department of Mathematics \& Statistics, Faculty of Science\\
	\textbf{Universiti Putra Malaysia}\\
	43400 UPM Serdang\\
	Selangor\\
	\textbf{Malaysia}}
\email{akilic@upm.edu.my}

\author[Malik]{\href{http://www.orcid.org/0000-0001-6350-575X}{Pradeep Malik}}
\address{Department of Mathematics\\
\textbf{Shree Guru Gobind Singh Tricentenary University}\\
Gurugram-122505, Haryana\\
\textbf{India}}
\email{pradeepmaths@gmail.com}

\author[Ayman-Mursaleen]{\href{https://orcid.org/0000-0002-2566-3498}{Mohammad Ayman-Mursaleen}*}
\address{
	School of Information \& Physical Sciences\\
	\textbf{The University of Newcastle}\\
	University Drive, Callaghan\\
	New South Wales 2308\\
	\textbf{Australia}\\
	Department of Mathematics \& Statistics, Faculty of Science\\
	\textbf{Universiti Putra Malaysia}\\
	43400 UPM Serdang\\
	Selangor\\
	\textbf{Malaysia}}
\email{mohammad.mursaleen@uon.edu.au, mursaleen.ayman@student.upm.edu.my, mohdaymanm@gmail.com}

\let\thefootnote\relax\footnotetext{*Corresponding author}


\begin{abstract}
In the present manuscript, we present a new sequence of operators, $i.e.$, $\alpha$-Bernstein-Schurer-Kantorovich operators depending on two parameters $\alpha\in[0,1]$ and $\rho>0$ for one and two variables to approximate  measurable functions on $[0: 1+q], q>0$. Next, we give basic results and discuss the rapidity of convergence and order of approximation for univariate and bivariate of these sequences in their respective sections. Further, Graphical and numerical analysis are presented. Moreover, local and global approximation properties are discussed in terms of first and second order modulus of smoothness, Peetre's K-functional and weight functions for these sequences in different spaces of functions.

\end{abstract}

\maketitle
\section{Introduction}
Bernstein (1912) \cite{ber-1} proposed the Bernstein polynomials as:
\begin{equation}\label{eqn1}
\mathbb{B}_l(h;v)=\sum\limits_{\nu=0}^{l}P_{l,\nu}(v)h\bigg(\frac{\nu}{l}\bigg), \textrm{ $l\in\mathbb{N}$},
\end{equation}
where $P_{l,\nu}(v)=\binom{l}{\nu}v^\nu(1-v)^{\nu-l}$. For the operators given by \eqref{eqn1}, he showed that $\mathbb{B}(g;v)$ converges to $g$ uniformly where $g\in C[0,1]$. In 1962, Schurer \cite{schurer} modified operators given in \eqref{eqn1} as: for $q>0$, a real number
 \begin{equation}\label{eqn11}
\mathbf{B}_{l,q}(g;v)=\sum\limits_{\nu=0}^{l+q}\binom{l+q}{\nu}v^\nu(1-v)^{\nu+q-l}g\bigg(\frac{\nu}{l}\bigg), \textrm{ $v\in [0,1+q]$},
\end{equation}
 where $g\in C[0,1+q]$. One can note that for $q=0$, the polynomials presented in \eqref{eqn11} reduces to polynomials given by \eqref{eqn1}. The operators are introduced in \eqref{eqn1} and \eqref{eqn11} are restricted for continuous functions only and are different in respect to the domain of function $f$. Several researchers, $e.g.$, Mursaleen et al. (\cite{as1}, \cite{as2},) Acar et al. (\cite{a1}, \cite{a2}), Mohiuddine et al. \cite{moh-1}, Ana et al. \cite{acu-1}, \.{I}\c{c}\"{o}z et al. (\cite{icoz}), \cite{icoz1}), Kajla et al. (\cite{kaj-1}, \cite{kaj-2})  constructed new sequences of linear positive operators to investigate the rapidity of convergence and order of approximation in different functional spaces in terms of several generating functions.
In the recent past, for $g\in[0, 1], m\in \mathbb{N}$ and $\lambda \in [-1, 1]$, Chen et al. \cite{che-1} constructed a sequence of new linear positive operators as:
\begin{equation}\label{eqn1.1}
T_{m,\lambda} (g;y)=\sum_{i=0}^{m}g\left(\frac{i}{m}\right)p_{m,i}^{\lambda}(y)   \qquad (y \in[0, 1]),
\end{equation}
where $p_{1,0}^{(\lambda)}=1-y$, \quad $p_{1,1}^{(\lambda)}=y$ and
\begin{eqnarray}\label{eqn1.2}
p_{m,i}^{\lambda}(y)&=&\left[(1-\lambda)y\binom{m-2}{i}+(1-\lambda)(1-y)\binom{m-2}{i-2}+\lambda y(1-y)\binom{m}{i}\right]\nonumber\\
&&y^{i-1}(1-y)^{m-i-1}   \quad (m\geq2).
\end{eqnarray}

The operators defined in \eqref{eqn1.1} are named as $\lambda-$Berstein operator of order m.\\
\begin{rem}
One can not that for $\lambda=1$, the relation (\ref{eqn1.1}) reduces to classical Bernstein operators \cite{ber-1}.
\end{rem}

  Later, Aral and Erbay \cite{ara-1} introduced a parametric extension of Baskakov operators. Recently, \"{O}zger et al. \cite{ozger2} constructed a sequence $\lambda$-Bernstein-Schurer operators as: For every $g\in C_B[0, \infty)$ where $C_B[0, \infty)$ stands for the continuous and bounded function, \begin{eqnarray}\label{eqn1.5}
\Psi_{n,\nu,\lambda}(g;u)&=&\sum_{k=0}^{n+\nu}g_kq_{n,\nu,k}^{(\lambda)}(u)
\end{eqnarray}
$$g_k=g\left(\frac{k}{n}\right)$$
Now, we construct the $\lambda-$ Bernstein- Schurer kantorovich operators and their moments\\
\begin{eqnarray}\label{eqn1.6}
\Psi_{n,\nu,\lambda}^*(g;u)&=&(n+1)\sum_{k=0}^{n+\nu}q_{n,\nu,k}^{(\lambda)}(u)\int_{\frac{k}{n+1}}^{\frac{k+1}{n+1}}g(s)ds
\end{eqnarray}
\begin{eqnarray*}\label{eqn1.41}
\mathcal{Q}_{n,k}^{(\lambda)}(u)&=&\frac{u^{k-1}}{(1+u)^{n+k-1}}\Bigg\{\frac{\lambda u}{1+u}\binom{n+k-1}{k}-(1-\lambda)(1+u)\binom{n+k-3}{k-2}\nonumber\\
&+&(1-\lambda)u\binom{n+k-1}{k}\Bigg\},
\end{eqnarray*}
with  $\binom{n-3}{-2}=\binom{n-2}{-1}=0$.  Motivating by the above development, we introduce positive linear operators  as follows:
\begin{eqnarray}\label{eqn1.411}
 K_{n,\lambda}^{\rho}(f;y) &=&\sum_{i=0}^{m+q} p_{m,i}^{\lambda}(y)\int_0^1g\left(\frac{i+t^{\rho}}{m+1}\right)ds,
\end{eqnarray}
 \begin{eqnarray*}\label{eqn1.2}
p_{m,i}^{\lambda}(y)&=&\Big[(1-\lambda)y\binom{m+q-2}{i}+(1-\lambda)(1-y)\binom{m+q-2}{i-2}\\
&+&\lambda y(1-y)\binom{m+q}{i}\Big]
y^{i-1}(1-y)^{m+q-i-1}   \quad (m\geq2).
\end{eqnarray*}

where $\rho>o$ and $\mathcal{Q}_{n,k}^{(\lambda)}(u)$ is given by \eqref{eqn1.41}.\\
In the subsequent sections, we investigate basic Lemmas, rate of convergence, order of approximation results. Locally and globally approximation results in terms of modulus of continuity, Peetre`s K-functional, second order modulus of smoothness, Lipschitz class and Lipschitz maximal function, weight functions. Further, $\lambda$-Bivariate Schurer Kantorovich operators are constructed and their pointwise and uniform approximation results are investigated.

\section{Basic Estimates}

\begin{lem}\label{lem1.1}
\cite{ozger2}For the operator defined in \eqref{eqn1.5}, one has
\begin{align*}
  \Psi_{n,\nu,\lambda}(e_0;u)&=1,\\
  \Psi_{n,\nu,\lambda}(e_1;u)&=u+\frac{\nu}{n}u,\\
  \Psi_{n,\nu,\lambda}(e_2;u)&=u^2+\frac{(n+\nu+2(1-\lambda))(u-u^2)}{n^2}+\frac{\nu(\nu+2n)u^2}{n^2}.\\
\end{align*}
\end{lem}

\begin{lem}\label{lem1.2}
Let $e_k(s)=s^k$, $k\in\{0,1,2\}$. Then, for the operators defined in \eqref{eqn1.6}, we have
\begin{align*}
  \Psi_{n,\nu,\lambda}^*(e_0;u)&=1,\\
  \Psi_{n,\nu,\lambda}^*(e_1;u)&=\bigg(\frac{n+\nu}{n+1}\bigg)u+\frac{1}{2(n+1)}\\
  \Psi_{n,\nu,\lambda}^*(e_2;u)&=\left[\frac{n^2}{(n+1)^2}-\frac{n+\nu+2(1-\lambda)}{(n+1)^2}+\frac{\nu(\nu+2n)}{(n+1)^2}\right]u^2\\
&+\left[\frac{n+\nu+2(1-\lambda)}{(n+1)^2}+\frac{n+\nu}{(n+1)^2}\right]u+\frac{1}{3(n+1)^2}.
\end{align*}
\end{lem}

\begin{lem}\label{lem1.2}
For the operator defined in \eqref{eqn1.411}, we have
\begin{align*}
K_{n,\lambda}^{\rho}(e_0;u)&=1,\\
K_{n,\lambda}^{\rho}(e_1;u)&=\frac{n+2(\lambda-1)}{n+1}u+\frac{(\lambda+1)(\rho+1)+1}{2(\rho+1)(n+1)},\\
K_{n,\lambda}^{\rho}(e_2;u)&=\left(1+\frac{4\lambda-3}{n}\right)\frac{n^2u^2}{(n+1)^2}\\
&+\frac{[(\rho+1)(n(2\lambda+3)+(\lambda-1)(2\lambda+7))]+4(\lambda-1)}{(\rho+1)(n+1)^2}u\\
&+\frac{2n(2\rho+1)+(\lambda+1)(2\rho+1)((\lambda+2)(\rho+1)+2)+\rho+1}{(2\rho+1)(\rho+1)(n+1)^2}.
\end{align*}
\end{lem}
\begin{proof}
Using Lemma \ref{lem1.2}, one can easily prove Lemma \ref{lem1.2}.
\end{proof}
\begin{lem}\label{lem1.3}
Let $e_{k}(s)=(e_1(s)-u)^k=\psi_u^k(s)$, $k\in {\mathbb N}$ be the central moments of $K_{n,\lambda}^{\rho}(.;.)$ constructed in \eqref{eqn1.411}. Then, we have
\begin{align*}
K_{n,\lambda}^{\rho}((e_1(s)-u);u)&=\frac{2\lambda-3}{n+1}u+\frac{(\lambda+1)(\rho+1)+1}{(\rho+1)(n+1)},\\
K_{n,\lambda}^{\rho}((e_1(s)-u)^2;u)&=\left[\left(1+\frac{4\lambda-3}{n}\right)\frac{n^2}{(n+1)^2}-\frac{2n+4\lambda-1}{n+1}+1\right]u^2\\
&+\frac{[(\rho+1)(n(2\lambda+3)+(\lambda-1)(2\lambda+7)-2(\lambda+1))]+\lambda-6}{(\rho+1)(n+1)^2}u\\
&+\frac{2n(2\rho+1)+(\lambda+1)(2\rho+1)((\lambda+2)(\rho+1)+2)+\rho+1}{(2\rho+1)(\rho+1)(n+1)^2}.
\end{align*}
\end{lem}

\begin{proof}
  Using Lemma 2.2, Lemma 2.3 can easily be proved.
\end{proof}

\section{Convergence behaviour of $K_{n,\lambda}^{\rho}(.;.)$}

\begin{definition}
\cite{moh-1} For $g\in C[0,1+q], q>0$, the modulus of continuity for a uniformly continuous function $g$ is defined as:

\begin{eqnarray*}
\omega(g;\delta)=\sup_{|r_1-r_2|\leq \delta}|g(r_1)-g(r_2)|, \hspace{1 cm} r_1,r_2\in[0,1+p], q>0.
\end{eqnarray*}
Let $g$ be a uniformly continuous function in $C[0,1+q], q>0$ and $\delta>0$. Then, one get
\begin{eqnarray}\label{3.1}
 |g(r_1)-g(r_2)|\leq \bigg(1+\frac{(r_1-r_2)^2}{\delta^2}\bigg)\omega(g;\delta).
\end{eqnarray}
\end{definition}

\begin{thm}\label{thm1}
Let $K_{n,\lambda}^{\rho}(.;.)$ be sequence of operators proposed by \eqref{eqn1.411}. Then,  $K_{n,\lambda}^{\rho}$ converges uniformly to $f$  on each bounded subset of $[0,1+q], q>0$ where $f\in C[0,1+q], q>0\bigcap\Bigg\{f: u\geq 0,  \frac{f(u)}{1+u^2}$ converges as $u\rightarrow \infty \Bigg\}.$
\end{thm}

\begin{proof}
To prove this result, it is adequate to prove that
\begin{eqnarray*}
K_{n,\lambda}^{\rho}(e_i;u)\rightarrow e_{i}(u), \textrm{ for } i\in\{0,1,2\}.
\end{eqnarray*}
Using Lemma \ref{lem1.2}, it is clear that $K_{n,\lambda}^{\rho}(e_i;u)\rightarrow e_{i}(u)$ for $i=0,1,2$  as $ n\rightarrow \infty$. Hence, Theorem \ref{thm1} is proved.
\end{proof}
\begin{ex}
One can note that, for the following set of parameters $q=5$, $\rho=0.1$ and $\lambda=0.5$, the operator $K_{n,\lambda}^{\rho}(f;x)$ converges uniformly to the function $f(y)=y^{3}-5y^{2}+6y+2$ as $n$ increases which is illustrated in the figure\ref{Fig-1}.
\begin{figure}[h!]
\centering
\includegraphics[width=3in]{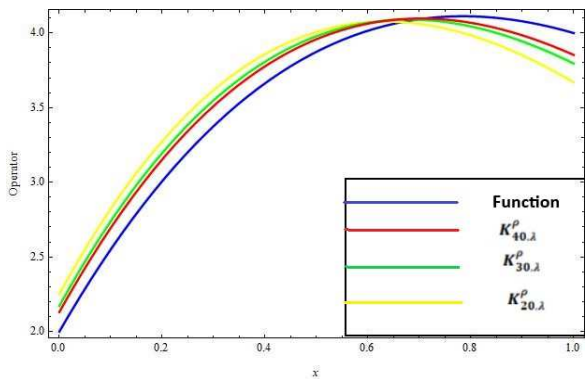}
\vspace{-0.4cm}
\caption{Approximation by operator $K_{n,\lambda}^{\rho}(;,;)$}\label{Fig-1}
\end{figure}
\begin{figure}[h!]
\includegraphics[width=3in]{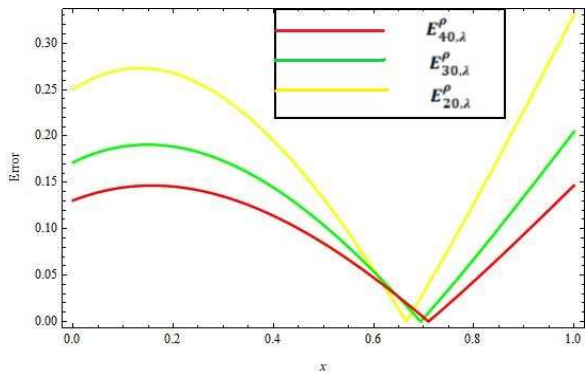}
\vspace{-0.4cm}
\caption{Error estimation of operators $K_{n,\lambda}^{\rho}(;,;)$ for the different values of $n$}\label{Fig-2}
\end{figure}
\begin{table}[h!]\label{tabb}
\begin{tabular}{|l|l|l|l|}
  \hline
  $x$ & $E_{20,\lambda}^{\rho}(f;x)$ & $E_{30,\lambda}^{\rho}(f;x)$ & $E_{40,\lambda}^{\rho}(f;x)$ \\ \hline
  0.1 & 0.2717372121 & 0.1887446733 & 0.1445360958 \\
  0.2 & 0.2677254718 & 0.1886482134 & 0.1455017073 \\
  0.3 & 0.2412358918 & 0.1732429202 & 0.1348677403 \\
  0.4 & 0.1951644878 & 0.1444179547 & 0.1140349291 \\
  0.5 & 0.1324072752 & 0.1040624783 & 0.0844040078 \\
  0.6 & 0.0558602697 & 0.0540656519 & 0.0473757106 \\
  0.7 & 0.0315805132 & 0.0036833631 & 0.0043507716 \\
  0.8 & 0.1270190580 & 0.0672954058 & 0.0432700748 \\
  0.9 & 0.2275593491 & 0.134881315  & 0.0940860947 \\
  1   & 0.3303053711 & 0.2045519293 & 0.1466965539 \\
  \hline
\end{tabular}
\caption{Error estimation table}
\end{table}
\end{ex}
\noindent The figure \ref{Fig-2} and Table 1 are also demonstrated our analytical results.\\
\begin{thm}\label{thm2}
(See \cite{shisha}) Let $L:C([a,b])\rightarrow B([a,b])$ be a linear and positive operator and let $\varphi_x$ be the function defined by

\begin{eqnarray*}
\varphi_x(t)=|t-x|, \textrm{ } (x,t)\in[a,b]\times[a,b].
\end{eqnarray*}
If $f\in C_B([a,b])$ for any $x\in[a,b]$ and any $\delta>0$, the operator $L$ verifies:
\begin{align*}
  |(Lf)(x)-f(x)|&\leq|f(x)||(Le_0)(x)-1|\\
  &+\{(Le_0)(x)+\delta^{-1}\sqrt{(Le_0)(x)(L\varphi_x^2)(x)}\}\omega_f(\delta).
\end{align*}
\end{thm}

\begin{thm}\label{thm3}
Let the operators $K_{n,\lambda}^{\rho}(.;.)$ be introduced by \eqref{eqn1.411} and $f\in C_B[0,1+p], q>0$, we have
\begin{eqnarray*}
|K_{n,\lambda}^{\rho}(f;u)-f(u)|\leq 2\omega(f;\delta),
\end{eqnarray*}
where $\delta=\sqrt{K_{n,\lambda}^{\rho}(\psi_u^2;u)}.$
\end{thm}

\begin{proof}
In view of Theorem \ref{thm2}, Lemma \ref{lem1.2} and Lemma \ref{lem1.3}, one has

\begin{eqnarray*}
|K_{n,\lambda}^{\rho}(f;u)-f(u)|\leq \{1+\delta^{-1}\sqrt{K_{n,\lambda}^{\rho}(f;u)(\psi_u^2;u)}\}\omega(f;\delta).
\end{eqnarray*}
On choosing $\delta=\sqrt{K_{n,\lambda}^{\rho}(\psi_u^2;u)}$, we completes the proof of this result.
\end{proof}

\section{Pointwise  Approximation results}

Here, we consider the Lipschitz type space \cite{ozarslan} as
\begin{eqnarray*}
Lip_M^{k_1,k_2}(\gamma)\!:=\!\Big\{f\!\in \!C_B[0,1+p], q>0 :|f(t)\!-\!f(u)|\!\leq\! M\frac{|t\!-\!u|^{\gamma}}{(t\!+\!k_1u\!+\!k_2u^2)^{\frac{\gamma}{2}}}:u,t\in(0,\infty)\Big\},
\end{eqnarray*}
where $M\geq0$ is a real valued constant number, $k_1,k_2>0,$ $\rho>0$ and $\gamma\in(0, 1]$.
\begin{thm}\label{thm3.3}
For $f\in Lip_M^{k_1,k_2}(\gamma)$, one yield
\begin{eqnarray}
|K_{n,\lambda}^{\rho}(f;u)-f(u)|\leq M\Bigg(\frac{\eta_n^*(u)}{k_1u+k_2u^2}\Bigg)^{\frac{\gamma}{2}},
\end{eqnarray}
where $u>0$ and $\eta_n^*(u)=K_{n,\lambda}^{\rho}(\psi_u^2;u)$.
\end{thm}

\begin{proof}
For $\gamma=1$, we have
\begin{align*}
|K_{n,\lambda}^{\rho}(f;u)-f(u)|&\leq K_{n,\lambda}^{\gamma}{(|f(t)-f(u)|)}(u)\\
&\leq M K_{n,\lambda}^{\rho}\left(\frac{|t-u|}{(t+k_1u+k_2u^2)^{\frac{1}{2}}};u\right).
\end{align*}
Since $\frac{1}{t+k_1u+k_2u^2}<\frac{1}{k_1u+k_2u^2}$ for all $t,u\in(0,\infty)$, we get
\begin{align*}
|K_{n,\lambda}^{\rho}(f;u)-f(u)|&\leq \frac{M}{(k_1u+k_2u^2)^{\frac{1}{2}}} (K_{n,\lambda}^{\rho}((t-u)^2;u))^{\frac{1}{2}}\\
&\leq M\Bigg(\frac{\eta_n^*(u)}{k_1u+k_2u^2}\Bigg)^{\frac{1}{2}}.
\end{align*}
This implies that for $\gamma=1,$ this result stand good. Now, for $\gamma\in (0,1)$ and using H\"{o}lder's inequality with $p=\frac{2}{\gamma}$ and $q=\frac{2}{2-\gamma}$, one obtain
\begin{align*}
|K_{n,\lambda}^{\rho}(f;u)-f(u)|&\leq \left(K_{n,\lambda}^{\rho}((|f(t)-f(u)|)^{\frac{2}{\gamma}};u)\right)^{\frac{\gamma}{2}}\\
&\leq M\left(K_{n,\lambda}^{\rho}\left(\frac{|t-u|^2}{(t+k_1u+k_2u^2)};u\right)\right)^{\frac{\gamma}{2}}.
\end{align*}
Since $\frac{1}{t+k_1u+k_2u^2}<\frac{1}{k_1u+k_2u^2}$ for all $t,u\in (0,\infty)$, we obtain
\begin{align*}
|K_{n,\lambda}^{\rho}(f;u)-f(u)|&\leq M \Bigg(\frac{\mathcal{P}_n^{\mu,\nu}\left(|t-u|^2;u\right)}{k_1u+k_2u^2}\Bigg)^{\frac{\gamma}{2}}\\
&\leq M \Big(\frac{\eta_n^*(u)}{k_1u+k_2u^2}\Big)^{\frac{\gamma}{2}}.
\end{align*}
Hence, we arrive at the desired result.
\end{proof}

\section{Global Approximation}

From \cite{gadjiev}, we recall some notation to prove the global approximation results.\\
For the weight function $1+u^2$ and $0\leq u<\infty$, we have\\
$B_{1+u^2}[0,1+p], q>0=\{f(u):|f(u)|\leq M_f (1+u^2),$ $M_f$  is constant depending on $f$\}.\\
$C_{1+u^2}[0,1+p], q>0\subset B_{1+u^2}[0,1+p], q>0$ space of continuous functions endowed with the norm $\|f\|_{1+u^2}=\sup\limits_{u\in[0,1+p], q>0}\frac{|f|}{1+u^2}$.\\
and
 $$C_{1+u^2}^{k}[0,1+p], q>0=\{f\in C_{1+u^2}: \lim\limits_{u\rightarrow\infty}\frac{f(u)}{1+u^2}=k, \textrm{ where } k \textrm{ is a constant}\}.$$

\begin{thm}\label{thm4.1}
Let the $K_{n,\lambda}^{\rho}(.;.)$ be the operators given by \eqref{eqn1.411} and $K_{n,\lambda}^{\rho}(.;.):C_{{1+u^2}}^k[0,1+p], q>0\rightarrow B_{1+u^2}[0,1+p], q>0.$ Then, we have
\begin{eqnarray*}
\lim\limits_{n\rightarrow\infty}\|K_{n,\lambda}^{\rho}(f;u)-f\|_{1+u^2}=0,
\end{eqnarray*}
where $f\in C_{{1+u^2}}^k[0,1+p], q>0.$
\end{thm}

\begin{proof}
To prove this result, it is sufficient to show that
\begin{align*}
\lim\limits_{n\rightarrow\infty}\|K_{n,\lambda}^{\rho}(e_i;u)-u^i\|_{1+u^2}=0, \textrm{   } i=0,1,2.
\end{align*}
From Lemma \ref{lem1.2}, we get
\begin{eqnarray*}
\|K_{n,\lambda}^{\rho}(e_0;u)-u^0\|_{1+u^2}=\sup\limits_{u\in[0,1+p], q>0}\frac{|K_{n,\lambda}^{\rho}(1;u)-1|}{1+u^2}=0 \textrm{ for }i=0.
\end{eqnarray*}
For $i=1$
\begin{align*}
\|K_{n,\lambda}^{\rho}(e_1;u)\!-\!u^1\|_{1+u^2}&\!=\!\sup\limits_{u\in[0,1+p], q>0}\frac{\frac{n+2(\lambda-1)}{n+1}u+\frac{(\lambda+1)(\rho+1)+1}{2(\rho+1)(n+1)}}{1+u^2}\\
&\!=\bigg(\frac{n+2(\lambda-1)}{n+1}-1\bigg)\sup\limits_{u\in[0,1+p], q>0}\frac{u}{1+u^2}\\
&+\frac{(\lambda+1)(\rho+1)+1}{2(\rho+1)(n+1)}\sup\limits_{u\in[0,1+p], q>0}\frac{1}{1+u^2}.
\end{align*}
Which implies that $\|K_{n,\lambda}^{\rho}(e_1;u)-u^1\|_{1+u^2}\rightarrow 0$ an $n\rightarrow \infty$.\\
Similarly, we see that  $\|K_{n,\lambda}^{\rho}(e_2;u)-u^2\|_{1+u^2}\rightarrow 0$ as $n\rightarrow \infty$.
 \end{proof}
\section{\textbf{Construction of  bivariate Sz\'{a}sz-Durrmeyer-Operators $H_{n_1,n_2}^{\ast}(.;. )$ and their Basic Estimates}}

Take $\mathcal{I}^2=\{(y_1,y_2):0 \leq y_1 < 1+q_1, \; 0 \leq y_2 < 1+q_2\}$ and $C\left(\mathcal{I}^2\right)$ is the class of all continuous functions on $\mathcal{I}^2$  equipped with the norm  $|| g||_{C\left(\mathcal{I}^2\right)}=\sup_{(y_1,y_2)\in\mathcal{I}^2 }|g(y_1,y_2)| $. Then for all $h\in C\left(\mathcal{I}^2\right)$ and $n_1,n_2\in \mathbb{N}$, we construct sequence of bivariate bivariate generalized baskakov operators as follows:
\begin{eqnarray}\label{eqn1.211}
 K_{m_1,m_2}^{\rho,\lambda_1,\lambda_2}(f;y_1,y_2) &=&\sum_{i_1=0}^{m_1+q_1} p_{m_1,i_1}^{\lambda_1}(y_1)\sum_{i_2=0}^{m_2+q_2} p_{m_2,i_2}^{\lambda_2}(y_2)\\ \nonumber
 &&\int_0^1\int_0^1g\left(\frac{i_1+t_1^{\rho}}{m_1+1},\frac{i_2+t_2^{\rho}}{m_2+1}\right)dt_1dt_2,
\end{eqnarray}
where
 \begin{eqnarray*}\label{eqn1.2}
 p_{m_k,i_k}^{\lambda_k}(y_k)&=&\Bigg[(1-\lambda)y_k\binom{m_k+q_k-2}{i_k}+(1-\lambda_k)(1-y_k)\binom{m_k+q_k-2}{i_k-2}\\
&+&\lambda y_k(1-y_k)\binom{m_k+q_k}{i_k}\Bigg]y_k^{i_k-1}(1-y_k)^{m_k+q_k-i_k-1}   \quad (m_1,m_2\geq2).
\end{eqnarray*}

\begin{lem}\label{lem2}
Let ${e_{i,j}}=y_1^iy_2^j$ be the central moments. Then, for the operators \eqref{eqn1.211}, we have
\begin{eqnarray*}
 K_{m_1,m_2}^{\rho,\lambda_1,\lambda_2}(e_{0,0};y_i,y_2)&=&1,\\
 K_{m_1,m_2}^{\rho,\lambda_1,\lambda_2}(e_{1,0};y_i,y_2)&=&\frac{m_1+2(\lambda_1-1)}{m_1+1}y_1+\frac{(\lambda_1+1)(\rho+1)+1}{2(\rho+1)(m_1+1)},\\
 K_{m_1,m_2}^{\rho,\lambda_1,\lambda_2}(e_{0,1};y_i,y_2)&=&\frac{m_1+2(\lambda_1-1)}{m_2+1}y_2+\frac{(\lambda_2+1)(\rho+1)+1}{2(\rho+1)(m_2+1)},\\
 K_{m_1,m_2}^{\rho,\lambda_1,\lambda_2}(e_{1,1};y_i,y_2)&=&\frac{m_1+2(\lambda_1-1)}{m_1+1}y_1+\frac{(\lambda_1+1)(\rho+1)+1}{2(\rho+1)(m_1+1)}\\
 &\times&\frac{m_2+2(\lambda_2-1)}{m_2+1}y_2+\frac{(\lambda_2+1)(\rho+1)+1}{2(\rho+1)(m_2+1)},\\
\end{eqnarray*}
\begin{eqnarray*}
 &&K_{m_1,m_2}^{\rho,\lambda_1,\lambda_2}(e_{2,0};y_i,y_2)=\left(1+\frac{4\lambda_1-3}{m_1}\right)\frac{m_1^2y_1^2}{(m_1+1)^2}\\
 &+&\frac{[(\rho+1)(m_1(2\lambda_1+3)+(\lambda_1-1)(2\lambda_1+7))]+4(\lambda_1-1)}{(\rho+1)(m_1+1)^2}y_1\\
&+&\frac{2m_1(2\rho+1)+(\lambda_1+1)(2\rho+1)((\lambda_1+2)(\rho+1)+2)+\rho+1}{(2\rho+1)(\rho+1)(m_1+1)^2},\\
 &&K_{m_1,m_2}^{\rho,\lambda_1,\lambda_2}(e_{0,2};y_i,y_2)=\left(1+\frac{4\lambda_2-3}{m_2}\right)\frac{m_2^2y_2^2}{(m_2+1)^2}\\
 &+&\frac{[(\rho+1)(m_2(2\lambda_2+3)+(\lambda_2-1)(2\lambda_2+7))]+4(\lambda_2-1)}{(\rho+1)(m_2+1)^2}y_2\\
&+&\frac{2m_2(2\rho+1)+(\lambda_2+1)(2\rho+1)((\lambda_2+2)(\rho+1)+2)+\rho+1}{(2\rho+1)(\rho+1)(m_2+1)^2}.\\
\end{eqnarray*}
\end{lem}
\begin{proof}
  In the light of Lemma \eqref{lem1.2} and linearity property, we have
  \begin{eqnarray*}
 K_{m_1,m_2}^{\rho,\lambda_1,\lambda_2}(e_{0,0};y_i,y_2)&=& K_{m_1,m_2}^{\rho,\lambda_1,\lambda_2}(e_0;y_i,y_2) K_{m_1,m_2}^{\rho,\lambda_1,\lambda_2}(e_0;y_i,y_2),\\
 K_{m_1,m_2}^{\rho,\lambda_1,\lambda_2}(e_{1,0};y_i,y_2)&=& K_{m_1,m_2}^{\rho,\lambda_1,\lambda_2}(e_1;y_i,y_2) K_{m_1,m_2}^{\rho,\lambda_1,\lambda_2}(e_0;y_i,y_2),\\
 K_{m_1,m_2}^{\rho,\lambda_1,\lambda_2}(e_{0,1};y_i,y_2)&=& K_{m_1,m_2}^{\rho,\lambda_1,\lambda_2}(e_0;y_i,y_2) K_{m_1,m_2}^{\rho,\lambda_1,\lambda_2}(e_1;y_i,y_2), \\
 K_{m_1,m_2}^{\rho,\lambda_1,\lambda_2}(e_{1,1};y_i,y_2)&=& K_{m_1,m_2}^{\rho,\lambda_1,\lambda_2}(e_1;y_i,y_2) K_{m_1,m_2}^{\rho,\lambda_1,\lambda_2}(e_1;y_i,y_2), \\
 K_{m_1,m_2}^{\rho,\lambda_1,\lambda_2}(e_{2,0};y_i,y_2)&=& K_{m_1,m_2}^{\rho,\lambda_1,\lambda_2}(e_2;y_i,y_2) K_{m_1,m_2}^{\rho,\lambda_1,\lambda_2}(e_0;y_i,y_2),\\
 K_{m_1,m_2}^{\rho,\lambda_1,\lambda_2}(e_{0,2};y_i,y_2)&=& K_{m_1,m_2}^{\rho,\lambda_1,\lambda_2}(e_0;y_i,y_2) K_{m_1,m_2}^{\rho,\lambda_1,\lambda_2}(e_2;y_i,y_2),
   \end{eqnarray*}
  which proves Lemma \eqref{lem2}.
\end{proof}
\begin{lem}\label{lem3}
  Let $\Psi_{y_1,y_2}^{i,j}(t,s)=\eta_{i,j}(t,s)=(t-y_1)^i(s-y_2)^j, \quad i,j\in\{0,1,2\}$ be the central moments functions.
Then from the operators $ K_{m_1,m_2}^{\rho,\lambda_1,\lambda_2}(.;.)$ defined by \eqref{eqn1.211} satisfies the following identities
  \begin{eqnarray*}
 K_{m_1,m_2}^{\rho,\lambda_1,\lambda_2}(\eta_{0,0};y_i,y_2)&=&1\\
 K_{m_1,m_2}^{\rho,\lambda_1,\lambda_2}(\eta_{1,0};y_i,y_2)&=&\frac{2\lambda_1-3}{m_1+1}y_1+\frac{(\lambda_1+1)(\rho+1)+1}{(\rho+1)(m_1+1)},\\
 K_{m_1,m_2}^{\rho,\lambda_1,\lambda_2}(\eta_{0,1};y_i,y_2)&=&\frac{2\lambda_1-3}{m_2+1}y_2+\frac{(\lambda_2+1)(\rho+1)+1}{(\rho+1)(m_2+1)}\\
 K_{m_1,m_2}^{\rho,\lambda_1,\lambda_2}(\eta_{1,1};y_i,y_2)&=&\frac{2\lambda_1-3}{m_1+1}y_1+\frac{(\lambda_1+1)(\rho+1)+1}{(\rho+1)(m_1+1)}\\
 &\times&\frac{2\lambda_1-3}{m_2+1}y_2+\frac{(\lambda_2+1)(\rho+1)+1}{(\rho+1)(m_2+1)}\\
  \end{eqnarray*}
  \begin{eqnarray*}
&& K_{m_1,m_2}^{\rho,\lambda_1,\lambda_2}(\eta_{2,0};y_i,y_2)=\left[\left(1+\frac{4\lambda_1-3}{m_1}\right)\frac{m_1^2}{(m_1+1)^2}-\frac{2m_1+4\lambda_1-1}{m_1+1}+1\right]y_1^2\\
&+&\frac{[(\rho+1)(m_1(2\lambda_1+3)+(\lambda_1-1)(2\lambda_1+7)-2(\lambda_1+1))]+\lambda_1-6}{(\rho+1)(m_1+1)^2}y_1\\
&+&\frac{2m_1(2\rho+1)+(\lambda_1+1)(2\rho+1)((\lambda_1+2)(\rho+1)+2)+\rho+1}{(2\rho+1)(\rho+1)(m_1+1)^2},\\
&& K_{m_1,m_2}^{\rho,\lambda_1,\lambda_2}(\eta_{0,2};y_i,y_2)=\left[\left(1+\frac{4\lambda_2-3}{m_2}\right)\frac{m_2^2}{(m_2+1)^2}-\frac{2m_2+4\lambda_2-1}{m_2+1}+1\right]y_2^2\\
&+&\frac{[(\rho+1)(m_1(2\lambda_2+3)+(\lambda_2-1)(2\lambda_2+7)-2(\lambda_2+1))]+\lambda_2-6}{(\rho+1)(m_2+1)^2}y_2\\
&+&\frac{2m_2(2\rho+1)+(\lambda_2+1)(2\rho+1)((\lambda_2+2)(\rho+1)+2)+\rho+1}{(2\rho+1)(\rho+1)(m_2+1)^2}.
\end{eqnarray*}
  \end{lem}
  \begin{proof}
     In the light of Lemma \eqref{lem2} and linearity property, we have
     \begin{eqnarray*}
 K_{m_1,m_2}^{\rho,\lambda_1,\lambda_2}(\eta_{0,0};y_i,y_2)&=& K_{m_1,m_2}^{\rho,\lambda_1,\lambda_2}(\eta_0;y_i,y_2) K_{m_1,m_2}^{\rho,\lambda_1,\lambda_2}(\eta_0;y_i,y_2),\\
 K_{m_1,m_2}^{\rho,\lambda_1,\lambda_2}(\eta_{1,0};y_i,y_2)&=& K_{m_1,m_2}^{\rho,\lambda_1,\lambda_2}(\eta_1;y_i,y_2) K_{m_1,m_2}^{\rho,\lambda_1,\lambda_2}(\eta_0;y_i,y_2),\\
 K_{m_1,m_2}^{\rho,\lambda_1,\lambda_2}(\eta_{0,1};y_i,y_2)&=& K_{m_1,m_2}^{\rho,\lambda_1,\lambda_2}(\eta_0;y_i,y_2) K_{m_1,m_2}^{\rho,\lambda_1,\lambda_2}(\eta_1;y_i,y_2), \\
 K_{m_1,m_2}^{\rho,\lambda_1,\lambda_2}(\eta_{1,1};y_i,y_2)&=& K_{m_1,m_2}^{\rho,\lambda_1,\lambda_2}(\eta_1;y_i,y_2) K_{m_1,m_2}^{\rho,\lambda_1,\lambda_2}(\eta_1;y_i,y_2), \\
 K_{m_1,m_2}^{\rho,\lambda_1,\lambda_2}(\eta_{2,0};y_i,y_2)&=& K_{m_1,m_2}^{\rho,\lambda_1,\lambda_2}(\eta_2;y_i,y_2) K_{m_1,m_2}^{\rho,\lambda_1,\lambda_2}(\eta_0;y_i,y_2),\\
 K_{m_1,m_2}^{\rho,\lambda_1,\lambda_2}(\eta_{0,2};y_i,y_2)&=& K_{m_1,m_2}^{\rho,\lambda_1,\lambda_2}(\eta_0;y_i,y_2) K_{m_1,m_2}^{\rho,\lambda_1,\lambda_2}(\eta_2;y_i,y_2),
   \end{eqnarray*}
  which proves Lemma \eqref{lem3}.
 \end{proof}

\section{\textbf{ Degree of Convergence}}

For any $g\in C(\mathcal{I}^2)$ and $\delta>0$  modulus of continuity of order second is given by

\begin{equation*}
\omega\left(g;\delta_{n_1},\delta_{n_2}\right)= \sup \{\mid g(t,s)-g(y_1,y_2)\mid : (t,s),(y_1,y_2)\in\mathcal{I}^2\}
\end{equation*}
with $\mid t-y_1\mid \leq \delta_{n_1},\;  \mid s-y_2\mid \leq \delta_{n_2}$ with the partial modulus of continuity defined as:
\begin{equation*}
\omega_1\left(g;\delta\right)= \sup_{0 \leq y_2\leq \infty}\sup_{\mid x_1-x_2 \mid \leq \delta }  \{\mid g(x_1,y_2)-g(x_2,y_2)\mid\},
\end{equation*}

\begin{equation*}
\omega_2\left(g;\delta\right)= \sup_{0 \leq y_1\leq \infty}\sup_{\mid y_1-y_2 \mid \leq \delta }  \{\mid g(y_1,y_1)-g(y_1,y_2)\mid\}.
\end{equation*}

\begin{thm}\label{new-5}
For any $g \in C(\mathcal{I}^2)$ we have
 \begin{equation*}
\mid  K_{m_1,m_2}^{\rho,\lambda_1,\lambda_2} (g;y_1,y_2)-g(y_1,y_2)\mid \leq 2\bigg(\omega_1(g;\delta_{y_1,n_1})+\omega_2(g;\delta_{n_2,y_2})\bigg).
\end{equation*}
\end{thm}

\begin{proof}
In order to give the prove of Theorem \ref{new-5}, in general we use well-known Cauchy-Schwarz inequality. Thus we see that
\begin{eqnarray*}
\mid  K_{m_1,m_2}^{\rho,\lambda_1,\lambda_2}(g;y_1,y_2)&-&g(y_1,y_2)\mid \leq   K_{m_1,m_2}^{\rho,\lambda_1,\lambda_2}\left(\mid g(t,s)-g(y_1,y_2)\mid;y_1,y_2 \right)\\
&\leq &  K_{m_1,m_2}^{\rho,\lambda_1,\lambda_2}\left(\mid g(t,s)-g(y_1,s)\mid;y_1,y_2 \right)\\&+&
  K_{m_1,m_2}^{\rho,\lambda_1,\lambda_2}\left(\mid g(y_1,s)-g(y_1,y_2)\mid;y_1,y_2 \right)\\
&\leq &  K_{m_1,m_2}^{\rho,\lambda_1,\lambda_2}\left(\omega_1(g;\mid t-y_1\mid);y_1,y_2 \right)\\
&+&  K_{m_1,m_2}^{\rho,\lambda_1,\lambda_2}\left(\omega_2(g;\mid s-y_2\mid);y_1,y_2 \right)\\
&\leq & \omega_1(g;\delta_{n_1})\left(1+\delta_{n_1}^{-1} K_{m_1,m_2}^{\rho,\lambda_1,\lambda_2}(\mid t-y_1\mid;y_1,y_2) \right)\\
&+&\omega_2(g;\delta_{n_2})\left(1+\delta_{n_2}^{-1} K_{m_1,m_2}^{\rho,\lambda_1,\lambda_2}(\mid s-y_2\mid;y_1,y_2) \right)\\
&\leq & \omega_1(g;\delta_{n_1})\left(1+\frac{1}{\delta_{n_1}} \sqrt{ K_{m_1,m_2}^{\rho,\lambda_1,\lambda_2} ((t-y_1)^2;y_1,y_2)}\right)\\
&+&\omega_2(g;\delta_{n_2})\left(1+\frac{1}{\delta_{n_2}} \sqrt{ K_{m_1,m_2}^{\rho,\lambda_1,\lambda_2} ((s-y_2)^2;y_1,y_2)}\right).
\end{eqnarray*}

If we choose $\delta_{n_1}^2=\delta_{n_1,y_1}^2= K_{m_1,m_2}^{\rho,\lambda_1,\lambda_2} ((t-y_1)^2;y_1,y_2)$ and $\delta_{n_2}^2=\delta_{n_2,y_2}^2= K_{m_1,m_2}^{\rho,\lambda_1,\lambda_2} ((s-y_2)^2;y_1,y_2),$ then we easily to reach our desired results.
\end{proof}

Here, we find convergence in terms of the Lipschitz class for bivariate function. For $M>0$ and $\tau,\tau \in[0,1+p], q>0,$  Lipschitz maximal function space on $E \times E \subset \mathcal{I}^2 $  defined by

\begin{eqnarray*}
\mathcal{L}_{\tau,\tau}(E\times E)&=&
\bigg{\{}g: \sup(1+t)^{\tau}(1+s)^{\tau}\left(g_{\tau,\tau}(t,s)-
g_{\tau,\tau}(y_1,y_2)  \right)\\
& \leq & M \frac{1}{(1+y_1)^{\tau}}\frac{1}{(1+y_2)^{\tau}}
 \bigg{\}},
\end{eqnarray*}
where $g$ is continuous and bounded on $\mathcal{I}^2 $, and

\begin{equation}\label{sn_1-1}
g_{\tau,\tau}(t,s)-g_{\tau,\tau}(y_1,y_2)= \frac{\mid g(t,s)-g(y_1,y_2)\mid }{\mid t-y_1 \mid^{\tau}\mid s-y_2 \mid^{\tau} }; \quad (t,s),(y_1,y_2)\in \mathcal{I}^2.
\end{equation}

\begin{thm}
\label{new-6} Let $g\in \mathcal{L}_{\tau,\tau}(E\times E),$ then for any $\tau,\tau \in [0,1+p], q>0,$  there exists $M>0$ such that

\begin{eqnarray*}
\mid  K_{m_1,m_2}^{\rho,\lambda_1,\lambda_2}(g;y_1,y_2)-g(y_1,y_2)\mid &\leq & M \bigg{\{}\bigg{(}\left(d(y_1,E)\right)^{\tau}+\left( \delta_{n_1,y_1}^2\right)^{\frac{\tau}{2}} \bigg{)}\\
&\times&\bigg{(}\left(d(y_2,E)\right)^{\tau}+\left( \delta_{n_2,y_2}^2\right)^{\frac{\tau}{2}} \bigg{)}\\
&+&\left(d(y_1,E)\right)^{\tau}  \left(d(y_2,E)\right)^{\tau}\bigg{\}},
\end{eqnarray*}
where $\delta_{n_1,y_1}$ and $\delta_{n_2,y_2}$ defined by Theorem \ref{new-5}.
\end{thm}

\begin{proof}
Take $\mid y_1-x_0 \mid =d(y_1, E)$ and $\mid y_2-y_0 \mid =d(y_2, E)$. For any $(y_1,y_2)\in \mathcal{I}^2$, and $(x_0,y_0)\in E\times E$ we let $d(y_1, E)=\inf \{ \mid y_1-y_2 \mid: y_2 \in E \}$. Thus we can write here
\begin{equation}\label{t-1}
\mid g(t,s)-g(y_1,y_2)\mid \leq M \mid g(t,s)-g(x_0,y_0)\mid + \mid g(x_0,y_0)-g(y_1,y_2)\mid.
\end{equation}%
Apply $ K_{m_1,m_2}^{\rho,\lambda_1,\lambda_2}$, we obtain
\begin{eqnarray*}
\mid  K_{m_1,m_2}^{\rho,\lambda_1,\lambda_2}(g;y_1,y_2)&-&g(y_1,y_2)\mid \\
&\leq & K_{m_1,m_2}^{\rho,\lambda_1,\lambda_2}\left( \mid g(y_1,y_2)-g(x_0,y_0)\mid + \mid g(x_0,y_0)-g(y_1,y_2)\mid \right)\\
&\leq & M  K_{m_1,m_2}^{\rho,\lambda_1,\lambda_2}\left(\mid t-x_0\mid^{\tau}\mid s-y_0\mid^{\tau};y_1,y_2 \right)\\
&+&M\mid y_1-x_0\mid^{\tau}\mid y_2-y_0\mid^{\tau}.
\end{eqnarray*}
For all $A,B \geq 0$ and $\tau \in [0,1+p], q>0$ we know inequality $(A+B)^{\tau}\leq A^{\tau}+B^{\tau},$ thus
\begin{equation*}
\mid t-x_0\mid^{\tau} \leq \mid t-y_1\mid^{\tau} +\mid y_1-x_0\mid^{\tau},
\end{equation*}

\begin{equation*}
\mid s-y_0\mid^{\tau} \leq \mid s-y_2\mid^{\tau} +\mid y_2-y_0\mid^{\tau}.
\end{equation*}
Therefore
\begin{eqnarray*}
\mid  K_{m_1,m_2}^{\rho,\lambda_1,\lambda_2}(g;y_1,y_2)-g(y_1,y_2)\mid
&\leq & M  K_{m_1,m_2}^{\rho,\lambda_1,\lambda_2}\left(\mid t-y_1\mid^{\tau}\mid s-y_2\mid^{\tau};y_1,y_2 \right)\\
&+&M\mid y_1-x_0\mid^{\tau} K_{m_1,m_2}^{\rho,\lambda_1,\lambda_2}\left( \mid s-y_2\mid^{\tau};y_1,y_2 \right)\\
&+&M\mid y_2-y_0\mid^{\tau} K_{m_1,m_2}^{\rho,\lambda_1,\lambda_2}\left( \mid t-y_1\mid^{\tau};y_1,y_2 \right)\\
&+& 2 M\mid y_1-x_0\mid^{\tau} \mid y_2-y_0\mid^{\tau} K_{m_1,m_2}^{\rho,\lambda_1,\lambda_2}\left(\mu_{0,0};y_1,y_2 \right).
\end{eqnarray*}
On apply H\"{o}lder inequality on $ K_{m_1,m_2}^{\rho,\lambda_1,\lambda_2}$, we get

\begin{eqnarray*}
  K_{m_1,m_2}^{\rho,\lambda_1,\lambda_2}\left(\mid t-y_1\mid^{\tau}\mid s-y_2\mid^{\tau};y_1,y_2 \right)& = & \mathcal{U}_{n_1,k}^{\lambda_1}\left(\mid t-y_1\mid^{\tau};y_1,y_2 \right) \\
  &\times&\mathcal{V}_{n_2,l}^{\lambda_2}\left(\mid s-y_2\mid^{\tau};y_1,y_2 \right) \\
 &\leq & \left( K_{m_1,m_2}^{\rho,\lambda_1,\lambda_2} (\mid t-y_1\mid^{2};y_1,y_2)\right)^{\frac{\tau}{2}}\\
&\times& \left( K_{m_1,m_2}^{\rho,\lambda_1,\lambda_2}(\mu_{0,0};y_1,y_2 ) \right)^{\frac{2-\tau}{2}}\\
 &\times &  \left( K_{m_1,m_2}^{\rho,\lambda_1,\lambda_2} (\mid s-y_2\mid^{2};y_1,y_2)\right)^{\frac{\tau}{2}}\\
 &\times &\left( K_{m_1,m_2}^{\rho,\lambda_1,\lambda_2}(\mu_{0,0};y_1,y_2 ) \right)^{\frac{2-\tau}{2}}.
\end{eqnarray*}
Thus, we can obtain
\begin{eqnarray*}
\mid  K_{m_1,m_2}^{\rho,\lambda_1,\lambda_2}(g;y_1,y_2)-g(y_1,y_2)\mid
&\leq & M\left( \delta_{n_1,y_1}^2\right)^{\frac{\tau}{2}}
\left( \delta_{n_2,y_2}^2\right)^{\frac{\tau}{2}} \\
&+&2M \left(d(y_1,E)\right)^{\tau}\left(d(y_2,E)\right)^{\tau}\\
&+&M \left(d(y_1,E)\right)^{\tau} \left( \delta_{n_2,y_2}^2\right)^{\frac{\tau}{2}}+L \left(d(y_2,E)\right)^{\tau}\left( \delta_{n_1,y_1}^2\right)^{\frac{\tau}{2}}.
\end{eqnarray*}
We have complete the proof.
\end{proof}

\begin{ex}
It is observed that in this example that  for the following set of parameters $q=5$, $\rho=0.9$ and $\lambda=0.5$, the operator $K_{m_{1}, m_{2}}^{\rho,\lambda_{1},\lambda_{2}}(,;,)$ converges uniformly to the function $f(y)=y_{1}^{3}y_{2}^{2}$ (Blue) as $m_{1}=m_{2}=10$ (Green) and $m_{1}=m_{2}=20$ (Red) increases which is shown in the Figure \ref{figure-3}.
\begin{figure}[h!]
  \centering
  \includegraphics[width=3in]{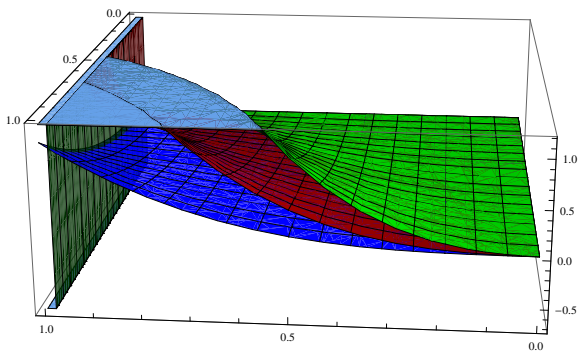}
  \caption{$K_{m_{1}, m_{2}}^{\rho,\lambda_{1},\lambda_{2}}(,;,)$ converges to $f(x)=y_{1}^{3}y_{2}^{2}$}\label{figure-3}
\end{figure}
\end{ex}

%\begin{ex}
%For the values of $\tau=0.5, \lambda=0.9$ and $f(x)=x^3-5x^2+6x+2$, the sequence of operators $K_{n,\lambda}^{\tau}(.;.)$ given by \eqref{eqn1.211} converges to $f(x)$ for differen values of $n$ as:
%\begin{figure}
%  \centering
%  \includegraphics[width=3.5in]{Fig1}
%  \caption{$K_{n,\lambda}^{\tau}(,;,)$ converges to $f(x)$ for $n=20, 40, 60$}\label{}
%\end{figure}
%\end{ex}
%\begin{ex}
%In the graph, we investigate error $E_{n,\lambda}^{\tau}$ behaviour for the operators introduced by \eqref{eqn1.211}, $\tau=0.5, \lambda=0.9$ as:
%\begin{figure}
%  \centering
%  \includegraphics[width=3.5in]{Fig2}
%  \caption{}\label{}
%\end{figure}
%\end{ex}
%\begin{ex}
%Here, we discuss numerical behaviour for different values of $x$ as:
%\begin{table}
%\begin{tabular}{|l|l|l|l|}
%  \hline
%  % after \\: \hline or \cline{col1-col2} \cline{col3-col4} ...
%  $x$ & $E_{20,\lambda}^{\tau}(f;x)$ & $E_{40,\lambda}^{\tau}(f;x)$ & $E_{60,\lambda}^{\tau}(f;x)$ \\ \hline
%  0.3 & 0.0229166396 & 0.0130876655 & 0.0091142386 \\
%  0.6 & 0.1364699277 & 0.0730024230 & 0.0497954983 \\
%  0.9 & 0.1605849260 & 0.0860739832 & 0.0587509659 \\
%  1.2 & 0.1000371450 & 0.0547116263 & 0.0375914988 \\
%  1.5 & 0.0403979051 & 0.0186753674 & 0.0120720456 \\
%  1.8 & 0.2559447144 & 0.1316777179 & 0.0886288103 \\
%  2.1 & 0.5418277724 & 0.2818861450 & 0.1904679378 \\
%  2.4 & 0.8932715689 & 0.4668913684 & 0.3159785708 \\
%  2.7 & 1.305500594  & 0.6842841079 & 0.4635498522 \\
%  3   & 1.773739337  & 0.9316550833 & 0.6315709244 \\
%  \hline
%\end{tabular}

%\section{Declarations}

\noindent \textbf{Acknowledgements}\newline
The third and the corresponding authors would like to acknowledge that this research was partially supported by Ministry of Higher Education under the Fundamental Research Grants Scheme (FRGS) with project number FRGS/2/2014/SG04/UPM/01/1 and having vot number 5524674. Also, the corresponding author is supported by the Jointly Awarded Research Degree (JADD) program (by UPM-UON) \newline

\noindent \textbf{Data Availability}\newline
Data sharing is not applicable to this article as no data sets were
generated or analyzed during the current study.\newline

\noindent \textbf{Conflicts of Interest}\newline
The authors declare that they have no conflicts of interest.\newline

\noindent \textbf{Funding}\newline
Not applicable.\newline

\noindent \textbf{Authors Contributions}\newline
All authors have equal contributions. \newline

% ------------------------------------------------------------------------
\end{document}